\documentclass[12pt]{article}
\usepackage{textgreek}
\usepackage[english]{babel}
\usepackage{graphicx}
\usepackage{caption}
\usepackage{float}
\usepackage[normalem]{ulem}
\useunder{\uline}{\ul}{}
\usepackage{rotating}
\usepackage{gensymb}
\usepackage{tikz}
\usepackage{fancyhdr}
\usepackage{amssymb}
\usepackage{amsmath}
\usepackage{amsthm}
\usepackage{amsfonts}
\usepackage[all]{xy}
\usepackage{mathrsfs}
\usepackage{latexsym}
\usepackage{mathdots}
\usepackage{setspace}
\usepackage{hyperref}
\hypersetup{
    colorlinks=true,
    linkcolor=blue,
    filecolor=blue,      
    urlcolor=blue,
    citecolor=red,
}
\usepackage{url}
\usepackage[utf8x]{inputenc}
\usepackage{graphicx}
\graphicspath{{images/}}
\usepackage{parskip}
\usepackage{fancyhdr}
\usepackage[a4paper, margin=0.9in]{geometry}
\usepackage{tikz-cd}
\usepackage{tikz}
\usetikzlibrary{chains}
\usepackage{enumitem}
\usepackage{nicematrix}

\DeclareMathAlphabet{\mathcal}{OMS}{cmsy}{m}{n}
\DeclareFontFamilySubstitution{LGR}{\rmdefault}{cm}
\newtheorem{theorem}{Theorem}[section]
\newtheorem{lemma}[theorem]{Lemma}
\newtheorem{prop}[theorem]{Proposition}

\newtheorem{coro}[theorem]{Corollary}
\theoremstyle{definition}
\newtheorem{defn}[theorem]{Definition}

\theoremstyle{remark}
\newtheorem{remark}[theorem]{Remark}
\title{On the Chow Rings of the Moduli Spaces $\mathcal{M}_{5,8}$ and $\mathcal{M}_{5,9}$}								
\author{Yuhan Liu}\date{}				
\makeatletter
\let\thetitle\@title
\let\theauthor\@author
\makeatother

\begin{document}
\hyphenpenalty=5000
\tolerance=1000
\title{On the Chow Rings of the Moduli Spaces $\mathcal{M}_{5,8}$ and $\mathcal{M}_{5,9}$}

\author{Yuhan Liu}

\maketitle

\begin{abstract}
\noindent In this paper we prove that the rational Chow rings of $\mathcal{M}_{5,8}$ and $\mathcal{M}_{5,9}$ are tautological, and that these moduli spaces have the Chow–K\"unneth generation Property.
\end{abstract}
%\pagebreak

\section{Introduction}

The moduli space of curves of genus $g$ has been a central topic in algebraic geometry over the past century. It has a natural compactification  $\overline{\mathcal{M}}_g$ by stable curves. Normalizing the singularities of the stable curves gives us curves with marked points. Therefore, the structure of the boundary of $\overline{\mathcal{M}}_g$ leads us to consider $\mathcal{M}_{g,n}$, which parametrizes moduli spaces of curves of genus $g$ with $n$ marked points. These $n$ marked points are ordered: One motivation for ordering the points is that different order corresponds to different glueing data.

One of the natural questions about moduli spaces $\mathcal{M}_{g,n}$ is to determine the Chow ring of $\mathcal{M}_{g,n}$, which has received considerable attention in the past 50 years. In 1983, David Mumford determined the rational Chow ring $A^*(\overline{\mathcal{M}}_2)$ \cite{M2}. (The integral Chow ring $A^*(\overline{\mathcal{M}}_2)$ is determined by Eric Larson in 2021 \cite{M2i}, but we are going to focus on rational Chow rings in this paper). In 1990, Carel Faber determined the rational Chow ring $A^*(\overline{\mathcal{M}}_3)$ \cite{M3}. Furthermore, substantial progress has been made on rational Chow rings $A^*(\mathcal{M}_g)$ for $g \leq 9$ \cite{M4} \cite{M5} \cite{M6} \cite{M789}.

There is a subring of the Chow ring called the tautological ring, whose structure is very well-understood. So it is natural to ask when the Chow ring is the same as the tautological ring, in which case we say the Chow ring is tautological. 

\begin{defn}
Let $f: \mathcal{C}_{g,n} \rightarrow \mathcal{M}_{g,n}$ be the universal curve, which comes equipped with $n$ sections $\sigma_i: \mathcal{M}_{g,n} \rightarrow \mathcal{C}_{g,n}$, corresponding to the $i$-th marked point.
We define $\psi$ classes and $\kappa$ classes as $$\psi_i=\sigma_i^{*}c_1(w_f),$$ $$\kappa_i =f_*\left(c_1(w_f)^{(i+1)}\right).$$ We call these classes tautological classes of $\mathcal{M}_{g,n}$ and we call the subring generated by these classes the tautological ring.
We also define $\lambda$ classes as $$\lambda_i=c_i\left(f_*w_f\right).$$
Note that by the Grothendieck Riemann-Roch Theorem, we can prove that $\lambda$ classes can be expressed in terms of $\kappa$ classes. In particular, the $\lambda$ classes are tautological.
\end{defn}
%\medbreak
\begin{defn}(Definition 3.1 of \cite{HS})
We say $Y$ has the Chow–K\"unneth generation Property (CKgP, for short) if for all algebraic stacks $X$ (of finite type and admitting a stratification by global quotient stacks), the exterior product map  
$$A_*(Y) \otimes A_*(X) \rightarrow A_*(Y \times X)$$
is surjective.
\end{defn}
It is useful to know when the moduli spaces $\mathcal{M}_{g,n}$ have the CKgP, due to the inductive structure of the boundary.

In \cite{M4} \cite{M5} \cite{M6} \cite{M789}, the authors have proved that the rational Chow ring of $A^*(\mathcal{M}_{g})$ is tautological for $4 \leq g \leq 9$. Furthermore, Canning and Larson proved $A^*(\mathcal{M}_{3,n})$ is tautological for $n \leq 11$; $A^*(\mathcal{M}_{4,n})$ is tautological for $n \leq 11$; $A^*(\mathcal{M}_{5,n})$ is tautological for $n \leq 7$; $A^*(\mathcal{M}_{6,n})$ is tautological for $n \leq 5$ \cite{HS}. Moreover, Canning and Larson proved that these moduli spaces have the CKgP \cite{HS}.

In the current paper, we push the result further when $g=5$, in which case the rational Chow ring is known to be tautological up to $n=7$ by Canning and Larson \cite{HS}.
\medskip

\begin{theorem}
The rational Chow rings of $\mathcal{M}_{5,8}$ and $\mathcal{M}_{5,9}$ are tautological, and these moduli spaces satisfy the CKgP.  
\end{theorem}

For any tetragonal smooth curve of genus 5, we consider its canonical embedding in $\mathbb{P}^4$. Any $n$ points on this curve always impose independent conditions on quadrics in $\mathbb{P}^4$ when $n$ is at most $7$. This is why the previous method showing the Chow ring is tautological breaks down when $n \geq 8$.
The key new innovation of this paper is to classify those configurations of 8 and 9 points that don't impose independent conditions, and prove that the loci of such marked curves have fundamental classes and Chow rings that are tautological.

\textbf{Idea of the proof.}
In the paper by Samir Canning and Hannah Larson \cite{HS}, they proved that all classes in $\mathcal{M}_{5,n}$ supported on the hyperelliptic locus are tautological for $n \leq 16$, and all classes in $\mathcal{M}_{5,n}$ supported on the trigonal locus are tautological for $n \leq 12$. By excision, it suffices to show that the Chow ring of the open locus $\mathcal{M}_{5,n} \setminus \mathcal{M}_{5,n}^3$ in $\mathcal{M}_{5,n}$ is tautological for $n=8,9$, where $\mathcal{M}_{g,n}^k$ is the locus of curves of gonality $\leq k$. This locus parametrizes curves which are complete intersections of three quadrics in $\mathbb{P}^4$ under the canonical embedding. Therefore, it is almost a Grassmann bundle over the configuration space of $n$ points, namely $(\mathbb{P}^4)^n$. It is not exactly a Grassmann bundle since the $n$ points may not impose independent conditions on quadrics in $\mathbb{P}^4$. We thus need to cut out the configurations of $n$ points according to their failure to impose independent condition on quadrics. For the $n$-pointed curves in the locus $\mathcal{M}_{5,n} \setminus \mathcal{M}_{5,n}^3$ in $\mathcal{M}_{5,n}$, we will see that for $n=7$, the $n$ marked points will always impose independent conditions on quadrics in $\mathbb{P}^4$; for $n=8$, the $8$ marked points will impose independent conditions on quadrics in $\mathbb{P}^4$ unless these 8 points sum up to the canonical bundle; for $n=9$, the $9$ marked points will impose independent conditions on quadrics in $\mathbb{P}^4$ unless $8$ of the $9$ points sum up to the canonical bundle.

\paragraph{Notation.} Throughout the paper, we use $A^*(\cdot)$ to represent the Chow ring with $\textit{rational}$ coefficients.
\paragraph{Convention.} For any vector bundle $\mathcal{K}$, we define its projectivization $\mathbb{P}\mathcal{K}:=\mathscr{P}\mathrm{roj}(Sym^\bullet \mathcal{K}^\vee)$. 
\paragraph{Characteristic hypothesis.} We work over an algebraically closed field of characteristic not 2, 3 or 5.

\paragraph{Acknowledgments.} I would like to thank my PhD advisor Eric Larson for helpful guidance throughout this project and feedback on earlier drafts of this manuscript. I would also like to thank
Samir Canning and Hannah Larson for feedback on an earlier draft of this paper.

\section{Independent locus in $\mathcal{M}_{5,n} \setminus \mathcal{M}_{5,n}^3$ for $n \leq 12 $}
 
More rigorously, the open locus $\mathcal{M}_{5,n}\setminus\mathcal{M}_{5,n}^3$ we mentioned above is a stack whose objects over a scheme $S$ are given by the following commutative diagrams:
% https://q.uiver.app/#q=WzAsMyxbMCwwLCJDIl0sWzIsMCwiUCJdLFswLDIsIlMiXSxbMCwxLCJqIiwwLHsic3R5bGUiOnsidGFpbCI6eyJuYW1lIjoiaG9vayIsInNpZGUiOiJ0b3AifX19XSxbMCwyLCJmIl0sWzEsMiwiXFxwaSJdLFsyLDAsIlxcc2lnbWFfMSwuLi4sXFxzaWdtYV9uIiwwLHsib2Zmc2V0IjotMywiY3VydmUiOi0yfV1d
\[\begin{tikzcd}
	C && P \\
	\\
	S
	\arrow["j", hook, from=1-1, to=1-3]
	\arrow["f", from=1-1, to=3-1]
	\arrow["\pi", from=1-3, to=3-1]
	\arrow["{\sigma_1,...,\sigma_n}", shift left=3, bend left, from=3-1, to=1-1]
\end{tikzcd}\]
where $f:C \rightarrow S$ is a smooth proper relative curve with $n$ pairwise disjoint sections $\sigma_1,...,\sigma_n : S \rightarrow C$; $\pi: P \rightarrow S$ is a $\mathbb{P}^4$-fibration; $j: C \hookrightarrow P$ is a closed embedding such that for every geometric point $s \in S$, $C_s \hookrightarrow \mathbb{P}_{\kappa(s)}^4$ is of degree $8$ via the canonical embedding. The morphisms in $\mathcal{M}_{5,n}\setminus\mathcal{M}_{5,n}^3$ between objects $\left(C \rightarrow P \rightarrow S, \sigma_1,...,\sigma_n:S \rightarrow C\right)$ and $\left(C' \rightarrow P' \rightarrow S, \sigma_1',...,\sigma_n':S \rightarrow C'\right)$ are isomorphisms $P \rightarrow P'$ inducing isomorphisms $C \rightarrow C'$ sending the sections $\sigma_i$ to $\sigma_i'$. We define the independent locus $U_n$ to be the open substack of $\mathcal{M}_{5,n}\setminus\mathcal{M}_{5,n}^3$ with the extra condition that $\sigma_1,...,\sigma_n : S \rightarrow C$ imposes independent conditions on quadrics. 
%In other words, $U_n$ is the locus where the sections impose independent condition on quadrics in $\mathbb{P}^4$. 
Note that the open substack $U_n$ admits a natural morphism to $BPGL_5$, sending the family of embedded curves to its associated $\mathbb{P}^4$-fibration. We define the stack $\mathcal{M}_{5,n}'$ by the following Cartesian diagram 
\[\begin{tikzcd}
	{\mathcal{M}_{5,n}'} && {U_n} \\
	\\
	{BSL_5} && {BPGL_5}
	\arrow[from=1-1, to=1-3]
	\arrow[from=1-1, to=3-1]
	\arrow[from=1-3, to=3-3]
	\arrow[from=3-1, to=3-3]
\end{tikzcd}\]

The stack $\mathcal{M}_{5,n}'$ is a $\mu_5$ gerbe over $U_n$. Thus $A^*(
\mathcal{M}_{5,n}') \cong A^*(U_n)
$. Furthermore, the points of $\mathcal{M}_{5,n}'$ over a scheme $S$ are given by diagrams 
% https://q.uiver.app/#q=WzAsMyxbMCwwLCJDIl0sWzIsMCwiXFxtYXRoYmJ7UH1WIl0sWzAsMiwiUyJdLFswLDEsImoiLDAseyJzdHlsZSI6eyJ0YWlsIjp7Im5hbWUiOiJob29rIiwic2lkZSI6InRvcCJ9fX1dLFswLDIsImYiXSxbMSwyLCJcXHBpIl0sWzIsMCwiXFxzaWdtYV8xLC4uLixcXHNpZ21hX24iLDAseyJvZmZzZXQiOi0zLCJjdXJ2ZSI6LTJ9XV0=
\[\begin{tikzcd}
	C && {\mathbb{P}V} \\
	\\
	S
	\arrow["j", hook, from=1-1, to=1-3]
	\arrow["f", from=1-1, to=3-1]
	\arrow["\pi", from=1-3, to=3-1]
	\arrow["{\sigma_1,...,\sigma_n}", shift left=3, bend left, from=3-1, to=1-1]
\end{tikzcd}\]
where $V$ is a rank 5 vector bundle over $S$ with trivial first Chern class. Let $\mathcal{V}$ be the universal bundle over $BSL_5$. We have a natural map $\gamma:\mathbb{P}\mathcal{V} \rightarrow BSL_5$. Let $(\mathbb{P}\mathcal{V})^n$ be the fiber product of $n$ copies of $\gamma:\mathbb{P}\mathcal{V} \rightarrow BSL_5$ over $BSL_5$. We use $\eta_i$ to denote the $i$-th projection  from $(\mathbb{P}\mathcal{V})^n$ to $\mathbb{P}\mathcal{V}$. By our construction, there are $n$ natural maps $\mathcal{M}_{5,n}' \rightarrow \mathbb{P}\mathcal{V}$ corresponding to $\sigma_i$. By the universal property of the fibered product, we have a map $b: \mathcal{M}_{5,n}' \rightarrow (\mathbb{P}\mathcal{V})^n$. Therefore, we have a composition map $p: \mathcal{M}_{5,n}' \xrightarrow{b} (\mathbb{P}\mathcal{V})^n \xrightarrow{\eta_i} \mathbb{P}\mathcal{V} \xrightarrow{\gamma} BSL_5$. Since $(\mathbb{P}\mathcal{V})^n$ is defined by pullback, we have $\gamma \circ \eta_i = \gamma \circ \eta_j$ for any $i,j$. We can thus denote $\gamma \circ \eta_i$ by $\overline{\gamma}$. We then consider the evaluation map on $(\mathbb{P}\mathcal{V})^n$: \begin{equation}
\overline{\gamma}^*\gamma_*\mathcal{O}_{\mathbb{P}\mathcal{V}}(2)=\eta_i^*Sym^2\mathcal{V}^{\vee} \rightarrow \bigoplus^{n}_{i=1}\eta_i^*\mathcal{O}_{\mathbb{P}\mathcal{V}}(2).   
\label{eva1}\end{equation} 
Note that by Nakayama's lemma, the evaluation map is surjective if and only if it is surjective on fibers. If we take any point $\Gamma$ in $(\mathbb{P}\mathcal{V})^n$, which is a collection of $n$ points $p_1,...,p_n$, the fiber of the evaluation map over $\Gamma$ is the map $H^0(\mathbb{P}^4,\mathcal{O}_{\mathbb{P}^4}(2)) \rightarrow H^0(\Gamma,\mathcal{O}_{\mathbb{P}^4}(2)|_\Gamma)$. 

\begin{prop}
    Assume $\Gamma$ is a collection of $n$ distinct points in $\mathbb{P}^4$, say $p_1,...,p_n$, which lie on a curve $C$ that is a complete intersection of 3 quadrics in $\mathbb{P}^4$. For $n \leq 7$, the map $$H^0(\mathbb{P}^4,\mathcal{O}_{\mathbb{P}^4}(2)) \rightarrow H^0(\Gamma,\mathcal{O}_{\mathbb{P}^4(2)}|_\Gamma)$$ is always surjective; for $n=8$, the map is surjective if and only if $\omega_C \ncong \mathcal{O}_C(p_1+ \cdots +p_8)$; for $n=9$, the map is surjective if and only if $\omega_C \ncong \mathcal{O}_C(p_{i_1}+ \cdots +p_{i_8})$ for $\{p_{i_1}, \cdots, p_{i_8}\} \subset \{1,2, \cdots 9\}$, in other words, 9 such points don't impose independent conditions on quadrics in $\mathbb{P}^4$ if and only if 8 of them don't.
\label{prop 0.2}
\end{prop}
\begin{proof} By our assumption, the evaluation map $H^0(\mathbb{P}^4,\mathcal{O}_{\mathbb{P}^4}(2)) \rightarrow H^0(\Gamma,\mathcal{O}_{\mathbb{P}^4}(2)|_\Gamma)$ factors as \\ $H^0(\mathbb{P}^4,\mathcal{O}_{\mathbb{P}^4}(2)) \rightarrow H^0(C,\mathcal{O}_C(2)) \rightarrow H^0(\Gamma,\mathcal{O}_{\mathbb{P}^4}(2)|_\Gamma) \rightarrow 0$. By Max Noether's Theorem \cite{MN}, we know that the first map is surjective. It remains to show that the second map is surjective. To do so, we consider the exact sequence 
\begin{equation}
0 \rightarrow \mathcal{O}_C(2)(-\Gamma) \rightarrow \mathcal{O}_C(2) \rightarrow \mathcal{O}_C(2)|_\Gamma \rightarrow 0.  
\label{res}
\end{equation} 
After taking global sections, it suffices to show that $H^1(\mathcal{O}_C(2)(-\Gamma))=0$. By Serre duality, this is equivalent to showing that $H^0(\mathcal{O}_C(-2)(\Gamma)\otimes \omega_C)=0$. 

For $n\leq7$, the bundle $\mathcal{O}_C(-2)(\Gamma)\otimes \omega_C$ is of degree $-16+n+8<0$. Thus we always have $H^0(\mathcal{O}_C(-2)(\Gamma)\otimes \omega_C)=0$. 

For $n=8$, we have $\deg \mathcal{O}_C(-2)(\Gamma)\otimes \omega_C = 0$. Therefore, $H^0(\mathcal{O}_C(-2)(\Gamma)\otimes \omega_C)=0$ if and only if $\omega_C \ncong \mathcal{O}_C(p_1+ \cdots +p_8).$ 

For $n=9$, the line bundle $\mathcal{O}_C(-2)(\Gamma)\otimes \omega_C \cong \omega_C^\vee(p_1+ \cdots + p_9)$ has degree one, thus $h^0(\mathcal{O}_C(-2)(\Gamma)\otimes \omega_C)\neq 0$ if and only if $\mathcal{O}_C(-2)(\Gamma)\otimes \omega_C \cong \mathcal{O}(p)$ for some point $p \in C$. Equivalently, $\mathcal{O}(p_1+ \cdots + p_9) \cong \omega_C(p).$ Note that $p$ is a base point of $\omega_C \otimes \mathcal{O}(p)$ and base points of $\mathcal{O}(p_1+ \cdots +p_9)$ are contained in the set $ \{p_1, \cdots, p_9\}$. Therefore, $p=p_i$ for some $1 \leq i \leq 9$. Without loss of generality, we assume $p=p_9$. We thus have $\mathcal{O}(p_1+ \cdots +p_8) \cong \omega_C$. Therefore, 9 points don't impose independent conditions if and only if 8 of them don't, in which case these 8 points lie on a hyperplane. 
\end{proof}
%\medskip

\begin{coro}
    $U_n=\mathcal{M}_{5,n} \setminus \mathcal{M}_{5,n}^3$ for $n \leq 7$.
\end{coro}

\begin{coro}
    $U_8$ is the open locus in $\mathcal{M}_{5,8}\setminus\mathcal{M}_{5,8}^3$ where the $8$ points don't sum up to the canonical bundle.
    \label{u8}
\end{coro}

\begin{coro}
    $U_9$ is the open locus in $\mathcal{M}_{5,9}\setminus\mathcal{M}_{5,9}^3$ where no $8$ of the $9$ points sum up to the canonical bundle.
\end{coro}
Define $V_n \subset (\mathbb{P}\mathcal{V})^n$ as the locus over which the evaluation map $\eqref{eva1}$ is surjective. Furthermore, we know that the image of $\mathcal{M}_{5,n}'$ under $b$ is contained in $V_n \subset (\mathbb{P}\mathcal{V})^n$. Therefore, the kernel of \eqref{eva1} over $V_n$ is a rank $15-n$ vector bundle. We denote the kernel by $\mathcal{E}$, which parametrizes tuples $(f, V, p_1, \dots, p_n)$, where $(p_1, \dots, p_n) \in (\mathbb{P}{V})^n$, $f$ is a quadratic form on $\mathbb{P}V$ with $f(p_1)= \cdots = f(p_n) = 0$. We have a natural map $G(3,\mathcal{E}) \rightarrow V_n \subset (\mathbb{P}\mathcal{V})^n$. From now on, we consider the evaluation map \eqref{eva1} over $V_n$.

\begin{prop}
    We have the following composition map: $$\mathcal{M}_{5,n}' \xrightarrow{\cong} X \subset G(3,\mathcal{E}) \rightarrow V_n \subset (\mathbb{P}\mathcal{V})^n \rightarrow BSL_5.$$
\end{prop}

\begin{proof} It remains to show that $\mathcal{M}_{5,n}'$ is isomorphic to an open stack of $G(3,\mathcal{E})$. The basic idea is to construct maps from $\mathcal{M}_{5,n}'$ to $G(3,\mathcal{E})$ and from an open set $X \subset G(3,\mathcal{E})$ to $\mathcal{M}_{5,n}'$. For the first map, each point in $\mathcal{M}_{5,n}'$ is a genus 5 curve with $n$ marked points; we map it to the 3 dimensional subspace of spaces of quadrics vanishing along the curve and therefore vanishing along the $n$ points. For the second map, given a three dimension space of quadrics which vanish at $n$ points and intersect transversely, their common vanishing locus gives a curve with genus 5 with the $n$ marked points. This is the first time we construct this map, so we will give a rigorous proof as follows.

Recall that we have the exact sequence 
\begin{equation}
0 \rightarrow \mathcal{E} \rightarrow \overline{\gamma}^*\gamma_*\mathcal{O}_{\mathbb{P}\mathcal{V}}(2) \rightarrow \bigoplus_{i=1}^n \eta_i^*\mathcal{O}_{\mathbb{P}\mathcal{V}}(2) \rightarrow 0 .    
\label{2}\end{equation}
Denote the composition map $\eta_i \circ b$ by $b_i$, the universal curve over $\mathcal{M}_{5,n}'$ by $\mathcal{C}_n$, the structure map $\mathbb{P}(p^*\mathcal{V}) \rightarrow \mathcal{M}_{5,n}'$ by $a$, and the embedding $\mathcal{C}_n \hookrightarrow \mathbb{P}(p^*\mathcal{V})$ by $j'$. On the universal curve $\mathcal{C}_n$, we have the line bundles $j'^*\mathcal{O}_{\mathbb{P}(p^*\mathcal{V})}(1)$ and $\omega_f$, which are the same on fibers. Therefore, they differ by a line bundle, that is, $j'^*\mathcal{O}_{\mathbb{P}(p^*\mathcal{V})}(1) = \omega_f \otimes f^*\mathcal{L}_1$. By the push pull formula, we get $$f_*\omega_f\otimes \mathcal{L}_1=f_*j'^*\mathcal{O}_{\mathbb{P}(p^*\mathcal{V})}(1)= a_*j'_*j'^*\mathcal{O}_{\mathbb{P}(p^*\mathcal{V})}(1) =p^*\mathcal{V}^\vee.$$ Therefore, we get $c_1(\mathcal{L}_1)=-\frac{1}{5}\lambda_1$. Taking the symmetric square shows $$b^*\overline{\gamma}^*\gamma_*\mathcal{O}_{\mathbb{P}\mathcal{V}}(2)= b^*\eta_i^*\gamma^*\gamma_*\mathcal{O}_{\mathbb{P}\mathcal{V}}(2)=p^*Sym^2\mathcal{V}^\vee=Sym^2(f_*\omega_f\otimes \mathcal{L}_1).$$ Furthermore, we consider the exact sequence 
\begin{equation}
    0 \rightarrow S \rightarrow Sym^2(f_*\omega_f \otimes \mathcal{L}_1) \rightarrow f_*((\omega_f \otimes f^*\mathcal{L}_1)^{\otimes 2})=f_*(\omega_f^{\otimes2})\otimes \mathcal{L}_1^{\otimes 2} \rightarrow 0.
\label{SES}
\end{equation}
\break
\textit{Claim.} $b_i^*\mathcal{O}_{\mathbb{P}\mathcal{V}}(2)=\sigma_i^*((\omega_f\otimes f^*\mathcal{L}_1)^{\otimes 2})$.
\\
\textit{Proof of claim.} 
Denote $j' \circ \sigma_i$ by $\eta_i'$ and consider the diagram 
\[\begin{tikzcd}
& \mathbb{P}(f_*\omega_f^\vee)  \arrow[dr,"g"] \\
\mathcal{C}_n \arrow[ur,hook,"\iota"] \arrow[rr,hook,"j'"] \arrow[dr,"f"]
&& \mathbb{P}(p^*\mathcal{V}) \arrow[dl,bend left = 15, "a"'] \arrow[r]
& \mathbb{P}\mathcal{V} \arrow[dl] \\
&\mathcal{M}_{5,n}' \arrow[ul,bend left,"\sigma_i"] \arrow[ur, bend left = 16, "\eta_i'"] \arrow[r]
& BSL_5
\end{tikzcd}\]
We have the isomorphism $g: \mathbb{P}(f_*\omega_f^\vee) \rightarrow \mathbb{P}(p^*\mathcal{V})$, which induces an isomorphism of line bundles $g^*\mathcal{O}_{\mathbb{P}(p^*\mathcal{V})}(2)=g^*a^*\mathcal{L}_1^{\otimes 2} \otimes \mathcal{O}_{\mathbb{P}(f_*\omega_f^\vee)}(2)$.
We have $$\begin{aligned}
b_i^*\mathcal{O}_{\mathbb{P}\mathcal{V}}(2) = \eta_i'^* \mathcal{O}_{\mathbb{P}(p^*\mathcal{V})}(2) &= \sigma_i^*\iota^*g^*\mathcal{O}_{\mathbb{P}(p^*\mathcal{V})}(2) \\&=
\sigma_i^*\iota^*(g^*a^*\mathcal{L}_1^{\otimes 2} \otimes \mathcal{O}_{\mathbb{P}(f_*\omega_f^\vee)}(2)) \\&=\sigma_i^*(\omega_f^{\otimes2}) \otimes \mathcal{L}_1^{\otimes2} \\&= \sigma_i^*(\omega_f \otimes f^*\mathcal{L}_1)^{\otimes2}.
\end{aligned}$$
This completes the proof of the claim. \\
Note that the evaluation map $f_*((\omega_f \otimes f^*\mathcal{L}_1)^{\otimes2}) \rightarrow \displaystyle{\bigoplus_{i=1}^n \sigma_i^*((\omega_f \otimes f^*\mathcal{L}_1)^{\otimes2})}$ is surjective. Therefore, we have $S \subset b^*\mathcal{E}$. By the universal property of the Grassmannian, this corresponds to a unique map $\mathcal{M}_{5,n}' \rightarrow G(3, \mathcal{E})$.
We consider the diagram % https://q.uiver.app/#q=WzAsMyxbMCwwLCJcXG1hdGhjYWx7UDo9fSBHKDMsIFxcbWF0aGNhbHtFfSkgXFx0aW1lcyBcXG1hdGhiYntQfVxcbWF0aGNhbHtWfSJdLFsyLDAsIlxcbWF0aGJie1B9XFxtYXRoY2Fse1Z9Il0sWzAsMiwiRygzLFxcbWF0aGNhbHtFfSkiXSxbMCwxLCJcXHBpXzIiXSxbMCwyLCJcXHBpXzEiXV0=
\[\begin{tikzcd}
	{\mathcal{P:=} G(3, \mathcal{E}) \times_{BSL_5} \mathbb{P}\mathcal{V}} && {\mathbb{P}\mathcal{V}} \\
	\\
	{G(3,\mathcal{E})}
	\arrow["{\pi_2}", from=1-1, to=1-3]
	\arrow["{\pi_1}", from=1-1, to=3-1].
\end{tikzcd}\]
%where $\mathcal{P}$ is the fiber product of $G(3, \mathcal{E})$ and $\mathbb{P}\mathcal{V}$ over $BSL_5$. 
Let $\mathcal{S}$ be the universal subbundle of $G(3,\mathcal{E})$. We then have a morphism on $\mathcal{P}$: \begin{equation}
    \pi_1^* \mathcal{S} \otimes \pi_2^* \mathcal{O}_{\mathbb{P}\mathcal{V}}(-2) \rightarrow \pi_1^* \mathcal{E} \otimes \pi_2^* \mathcal{O}_{\mathbb{P}\mathcal{V}}(-2) \rightarrow \mathcal{O}_\mathcal{P} \label{5}
\end{equation}
where the first map arises from the tautological sequence on $G(3, \mathcal{E})$ and the second is obtained from multiplying forms. Take $\mathcal{C}$ to be the vanishing locus of \eqref{5}. Then we get an embedding $j'' : \mathcal{C} \hookrightarrow \mathcal{P}$. Moreover, we have the normal bundle $\mathcal{N}_{\mathcal{C}/\mathcal{P}} = \pi_1^* \mathcal{S}^\vee \otimes \pi_2^* \mathcal{O}_{\mathbb{P}\mathcal{V}}(2).$ By our construction, the fiber of $\pi_1$ restricted to $\mathcal{C}$ is the locus of the intersection of three quadrics. Next, we show that $ \mathcal{C} \rightarrow G(3, \mathcal{E})$ has $n$ sections. The identity map between $G(3, \mathcal{E})$ together with the composition $G(3, \mathcal{E}) \rightarrow V_n \subset (\mathbb{P}\mathcal{V})^n \xrightarrow{\eta_i} \mathbb{P}\mathcal{V}$ induce $n$ sections of $\pi_1$,
$$\sigma_i: G(3, \mathcal{E}) \rightarrow \mathcal{P=} G(3, \mathcal{E}) \times_{BSL_5} \mathbb{P}\mathcal{V}.$$
These sections factor through $\mathcal{C}$ by our construction. Indeed, the fiber of $G(3, \mathcal{E})$ over $\{p_1, \dots, p_n\} \in V_n$ is the subspace spanned by three quadrics which vanish at the $p_i$. The fiber of $\mathcal{C} \rightarrow G(3, \mathcal{E})$ (over the subspace spanned by three quadrics) is the vanishing locus of these three quadrics. 

Let $X$ be the open locus where the fiber of $\mathcal{C} \rightarrow G(3, \mathcal{E})$ is a smooth curve, which is the complete intersection of three quadrics. More explicitly, we consider the projection map $\pi_1: G(3,\mathcal{E})\times \mathbb{P}\mathcal{V}  \rightarrow G(3, \mathcal{E})$ and define a closed subset $W$ in $G(3,\mathcal{E})\times \mathbb{P}\mathcal{V}$ \begin{multline}
W:=\bigg\{(h_1,h_2,h_3,p)\in G(3,\mathcal{E})\times \mathbb{P}\mathcal{V} \; \bigg| \; h_1(p)=h_2(p)=h_3(p)=0 , \\ \text{ and all } 3 \times 3 \text{ minors of } \left(\frac{\partial h_i}{\partial x_j}\right)_{\substack{i=1,2,3 \\ j=1,2,3,4,5}} \text{ vanishes} \bigg\}.    
\end{multline}
Note that the vanishing equations give a closed condition, and the complement of $\pi_1(W)$ satisfies the property that the intersection of the three quadratics is a smooth complete intersection. Therefore, $\mathcal{M}_{5,n}'$ is isomorphic to $X \subset G(3, \mathcal{E})$. We thus have the following composition map: $$ \mathcal{M}_{5,n}' \xrightarrow{\cong} X \subset G(3,\mathcal{E}) \rightarrow V_n \subset (\mathbb{P}\mathcal{V})^n \rightarrow BSL_5.$$ 
\end{proof}

\begin{coro}
The Chow ring of $\mathcal{M}_{5,n}'$ has the same generators as the Chow ring of $G(3, \mathcal{E})$, which is generated by $c_1(\mathcal{S})$, $c_2(\mathcal{S})$, $c_3(\mathcal{S})$, $c_2(\mathcal{V})$, $c_3(\mathcal{V})$, $c_4(\mathcal{V})$, $c_5(\mathcal{V})$, and $\eta_i^*\mathcal{O}_{\mathbb{P}\mathcal{V}}(1) \text{ where }1 \leq i \leq n$.    
\end{coro}
\medskip

\begin{lemma}
  The classes $c_2(\mathcal{V})$, $c_3(\mathcal{V})$, $c_4(\mathcal{V})$, $c_5(\mathcal{V})$, and $\eta_i^*\mathcal{O}_{\mathbb{P}\mathcal{V}}(1)$ are tautological on $\mathcal{M}_{5,n}'$. 
\end{lemma} 
\begin{proof} Following our notation above, we have the universal diagram restricted to $X \subset G(3,\mathcal{E})$:
% https://q.uiver.app/#q=WzAsMyxbMCwwLCJcXG1hdGhjYWx7Q30iXSxbMiwwLCJcXG1hdGhjYWx7UH0iXSxbMCwyLCJYIl0sWzAsMSwiaiIsMCx7InN0eWxlIjp7InRhaWwiOnsibmFtZSI6Imhvb2siLCJzaWRlIjoidG9wIn19fV0sWzAsMiwiZiJdLFsxLDIsIlxccGlfMSJdLFsyLDAsIlxcc2lnbWFfMSwgXFxkb3RzLCBcXHNpZ21hXzciLDAseyJvZmZzZXQiOi0zLCJjdXJ2ZSI6LTJ9XV0=
\[\begin{tikzcd}
	{\mathcal{C}} && {\mathcal{P}} \\
	\\
	X
	\arrow["j''", hook, from=1-1, to=1-3]
	\arrow["f", from=1-1, to=3-1]
	\arrow["{\pi_1}", from=1-3, to=3-1]
	\arrow["{\sigma_1, \dots, \sigma_n}", shift left=3, bend left, from=3-1, to=1-1]
\end{tikzcd}\]

By adjunction, we have $$\begin{aligned}
w_f&=j''^*\left(w_{\pi_1} \otimes \det \left(\pi_1^*\mathcal{S}^{\vee} \otimes \pi_2^*\mathcal{O}_{\mathbb{P}\mathcal{V}}(2)\right)\right)\\&=j''^*\left(w_{\pi_1} \otimes \det (\pi_1^*\mathcal{S}^{\vee}) \otimes \pi_2^*\mathcal{O}_{\mathbb{P}\mathcal{V}}(6)\right)\\&=j''^*\left(\pi_2^*\mathcal{O}_{\mathbb{P}\mathcal{V}}(1) \otimes \det (\pi_1^*\mathcal{S}^{\vee})\right).
\end{aligned}$$

Pushing forward by $f$, we have $$f_*(\omega_f)=\det (\mathcal{S}^{\vee}) \otimes \pi_{1*}j''_*j''^*\pi_{2}^*\mathcal{O}_{\mathbb{P}\mathcal{V}}(1)=\det (\mathcal{S}^{\vee}) \otimes \mathcal{V}^{\vee}.$$

By taking the first Chern class, we see that $\lambda_1=c_1\left(f_*(\omega_f)\right)=5c_1(\det (\mathcal{S}^{\vee}))$. Taking higher Chern classes and using the splitting principle shows that $c_2(\mathcal{V})$, $c_3(\mathcal{V})$, $c_4(\mathcal{V})$, $c_5(\mathcal{V})$ are polynomials in the $\lambda$ classes. Meanwhile, pulling back by $\sigma_i$, we have $$\sigma_i^*w_f=\det (\mathcal{S}^{\vee})\otimes \sigma_i^*j''^*\pi_2^* \mathcal{O}_{\mathbb{P}\mathcal{V}}(1)=\det (\mathcal{S}^{\vee})\otimes\eta_i^*\mathcal{O}_{\mathbb{P}\mathcal{V}}(1).$$ Taking first Chern classes, we see that $\psi_i=c_1\left(\sigma_i^*w_f\right)=\frac{1}{5}\lambda_1+c_1\left(\eta_i^*\mathcal{O}_{\mathbb{P}\mathcal{V}}(1)\right)$. Thus $c_1\left(\eta_i^*\mathcal{O}_{\mathbb{P}\mathcal{V}}(1)\right)=\psi_i-\frac{1}{5}\lambda_1$. 
\end{proof}
\begin{lemma}
    The classes $c_1(\mathcal{S})$, $c_2(\mathcal{S})$, $c_3(\mathcal{S})$ are tautological on $\mathcal{M}_{5,n}'$.
\end{lemma}
\begin{proof}
By our short exact sequence \eqref{SES} and Whitney's Formula, we have $$c(\mathcal{S})=\frac{c\left(Sym^2(f_*w_f\otimes \mathcal{L}_1)\right)}{c\left(f_*(w_f^{\otimes 2}) \otimes \mathcal{L}_1^{\otimes 2}\right)}.$$  By Example 5.16 in \cite{3264}, $c\left(Sym^2f_*(w_f)\right)$ can be written in $\lambda_1$ and $\lambda_2$. By the Grothendieck Riemann-Roch Theorem, we have \begin{equation}
Ch\left(f_*(w_f^{\otimes 2})\right)-Ch\left(f_*(w_f^{\vee})\right)=f_*\left[Ch(w_f^{\otimes 2}) \cdot Td(w_f^{\vee})\right]. \label{GRR}   
\end{equation} Since $h^0(w_f^{\vee})=0$, we get $f_*(w_f^{\vee})=0$ by Grauert's Theorem. Moreover, any class in $\eqref{GRR}$ can be expressed in tautological classes. We thus have that $Ch\left(f_*(w_f^{\otimes 2})\right)$ all tautological. Therefore, $c(\mathcal{S})$ is tautological. 
\end{proof}
\begin{coro}
    The Chow ring $A^*(U_n)$ is tautological for $n \leq 12$. 
\label{coro25}
\end{coro}
%\medbreak
\begin{remark}
    From our definition of $U_n$, we have $U_n = \varnothing $ when $n > 12$. 
\end{remark}
% \medbreak
\begin{coro}
    In particular, the Chow ring of $\mathcal{M}_{5,7}'$ is tautological. Therefore, the Chow ring of $\mathcal{M}_{5,7}$ is tautological.
\end{coro}
\begin{proof}Consider the following two exact sequences:
\[\begin{tikzcd}
A^*(\mathcal{M}_{5,7}^2) \arrow[r,"i_*"] & A^*(\mathcal{M}_{5,7}) \arrow[r] & A^*(\mathcal{M}_{5,7}\setminus\mathcal{M}_{5,7}^2) \arrow[r] \arrow[dl,equal] & 0 \\
A^*(\mathcal{M}_{5,7}^3\setminus\mathcal{M}_{5,7}^2) \arrow[r,"i'_*"] & A^*(\mathcal{M}_{5,7}\setminus\mathcal{M}_{5,7}^2) \arrow[r] & A^*(\mathcal{M}_{5,7}\setminus\mathcal{M}_{5,7}^3) \arrow[r] & 0 
\end{tikzcd}\]
where $\mathcal{M}_{g,n}^k$ is the locus of curves of gonality $\leq k$.
By Theorem 6.1 and Lemma 9.9 in \cite{HS}, and Proposition 1 in \cite{T}, we have that the images of the pushforwards $i_*$ and $i'_*$ are tautological. Together with our conclusion that $A^*(\mathcal{M}_{5,7}')=A^*(\mathcal{M}_{5,7}\setminus\mathcal{M}_{5,7}^3)$ is tautological, we conclude that $A^*(\mathcal{M}_{5,7})$ is tautological. 
\end{proof}

As a consequence, we obtain an alternative argument that $A^*(\mathcal{M}_{5,7})$ is tautological, which was originally proven by Samir Canning and Hannah Larson in \cite{HS}.
%\medbreak

However, for $n=8$ and $9$, we cannot yet conclude that the Chow ring of $\mathcal{M}_{5,n}$ is tautological, because the locus $\mathcal{M}_{5,n} \setminus \mathcal{M}_{5,n}^3$ is not a Grassmann bundle over some open substack of the $n$-fold fiber product of the universal $\mathbb{P}^4$-fibration over $BPGL_5$. Indeed, certain configurations of $n$ points in $\mathbb{P}^4$ don't impose independent conditions on quadrics, even though there are smooth canonical curves with genus 5 passing through these $n$ points. So for $n=8$ and $n=9$, we need to prove that the loci of such marked curves have fundamental classes and Chow rings that are tautological.

\section{Classes supported on $(\mathcal{M}_{5,8} \setminus \mathcal{M}_{5,8}^3)\setminus U_8$}

By Proposition \ref{prop 0.2}, we know that $(\mathcal{M}_{5,8} \setminus \mathcal{M}_{5,8}^3)\setminus U_8$ parametrizes smooth curves of genus 5 with 8 marked points, such that the 8 marked points are the complete intersection of three quadrics and a hyperplane. Therefore, we define $\mathcal{M}_\omega:= (\mathcal{M}_{5,8} \setminus \mathcal{M}_{5,8}^3) \setminus U_8$ (as a substack) to be the locus where the evaluation map $f_*\omega_f \rightarrow \displaystyle{\bigoplus_{i=1}^8\sigma_i^*\omega_f}$ drops rank. We denote the universal curve over $\mathcal{M}_\omega$ by $\mathcal{C}_w$.
%\medbreak
\begin{prop}
    The fundamental class $[\mathcal{M}_\omega]$ is tautological.
    \label{prop04}
\end{prop}
\begin{proof} Recall that $\mathcal{M}_\omega$ is the locus where the map $f_*w_f \rightarrow \displaystyle{\bigoplus_{i=1}^8\sigma_i^*\omega_f}$ drops rank. Note that the cycle $[\mathcal{M}_\omega]$ has codimension 4 in $\mathcal{M}_{5,8} \setminus \mathcal{M}_{5,8}^3$. Indeed, consider the forgetful map $\mathcal{M}_\omega \rightarrow \mathcal{M}_5$; the fiber corresponds to a general element in the complete linear series of the canonical. Therefore, the dimension of $\mathcal{M}_\omega$ is $12+4=16$, which implies that $\mathcal{M}_\omega$ has codimension 4. Thus we can use Porteous' formula. Using the same notation as Theorem 12.4 in \cite{3264}, we have \begin{equation}
    [\mathcal{M}_\omega]
    =\Delta_4^1\left[\frac{(1+\psi_1t)\cdots (1+\psi_8t)}{1+\lambda_1t+\cdots + \lambda_5t^5}\right]
    =\Biggl\{\frac{(1+\psi_1t)\cdots (1+\psi_8t)}{1+\lambda_1t+\cdots + \lambda_5t^5}\Biggr\}^4. 
\label{3}\end{equation}
All terms involved in \eqref{3} are tautological, thus $[\mathcal{M}_\omega]$ is tautological. 
\end{proof}

It thus suffices to prove the Chow ring of $\mathcal{M}_\omega$ is generated by restriction of tautological classes on $\mathcal{M}_{5,8} \setminus \mathcal{M}_{5,8}^3$.

Note that in the locus $\mathcal{M}_\omega$, we have the further condition that the marked points lie on a hyperplane in $\mathbb{P}^4$. Therefore, we consider the following exact sequences:
\begin{equation}
    0 \rightarrow f_*\omega_f(-\sigma_1-\cdots-\sigma_8) \rightarrow f_*\omega_f \rightarrow \mathcal{G}' \rightarrow 0.
    \label{exact1}
\end{equation}
We denote the line bundle $f_*\omega_f(-\sigma_1-\cdots-\sigma
_8)$ by $\mathcal{L}$ for simplicity. Tensoring the above exact sequence \eqref{exact1} with $\mathcal{L}^\vee$, we get the normalized exact sequence:
\begin{equation}
0 \rightarrow \mathcal{O}_{\mathcal{M}_\omega} \rightarrow \mathcal{L}^\vee \otimes f_*\omega_f \rightarrow \mathcal{L}^\vee \otimes \mathcal{G}' \rightarrow 0.
\label{nexact}
\end{equation}
We denote $\mathcal{L}^\vee \otimes \mathcal{G}'$ by $\mathcal{G}$, and the natural map $\mathcal{M}_\omega \rightarrow \mathbb{P}\mathcal{G}^\vee$ corresponding to $\sigma_i$ by $\eta_i'.$ For the next step, we are going to parametrize these 8 points using $(\mathbb{P}^3)^7$, since in nice cases the eighth point is uniquely determined by the first seven points. 
\medskip
\begin{lemma}
    If 8 points $p_1, \dots, p_8$ are the complete intersection of 3 quadrics in $\mathbb{P}^3$, then $p_1, \dots, p_7$ must impose independent conditions on quadrics in $\mathbb{P}^3$.
\label{lemma31}
\end{lemma}

\begin{proof}
We will first find all the cases where 7 points do not impose independent conditions on quadrics. Then we will go through each of these cases, and see whether they can be the subset of a complete intersection of 3 quadrics.\\
Let $H$ be the plane containing the maximum number of $p_1, \cdots, p_7$. Denote the set of points lying on $H$ by $\Delta$ and the set of points not in $H$ by $\Sigma$.
%Denote the coordinates in $\mathbb{P}^3$ by  $x,y,z$ and $w$. \\
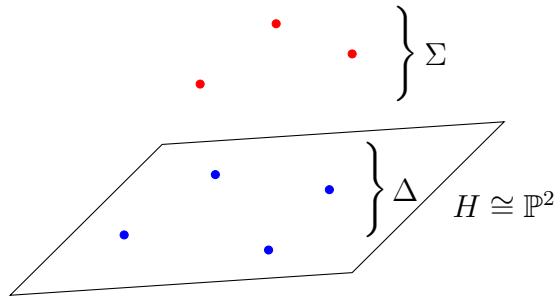
\[\begin{tikzpicture}
  % Define vertices
  \filldraw [red] (1,1.3) circle (1.5pt);
  \filldraw [red] (3,1.7) circle (1.5pt);
  \filldraw [red] (2,2.1) circle (1.5pt);
  \filldraw [blue] (0,-0.7) circle (1.5pt);
  \filldraw [blue] (1.2,0.1) circle (1.5pt);
  \filldraw [blue] (1.9,-0.9) circle (1.5pt);
  \filldraw [blue] (2.7,-0.1) circle (1.5pt);
  \coordinate (A) at (-1.5,-1.5);
  \coordinate (B) at (3,-1.2);
  \coordinate (C) at (5,0.8);
  \coordinate (D) at (0.5,0.5);
  % Draw the parallelogram
  \draw (A) -- (B) -- (C) -- (D) -- cycle;
  \node at (5,-0.3) {$H \cong \mathbb{P}^2$};
  \node at (3.7,1.7) {$\Biggl\}$};
  \node at (4.1,1.7) {$\Sigma$};
  \node at (3.3, -0.1) {$\Biggl\}$};
  \node at (3.7,-0.1) {$\Delta$};
\end{tikzpicture}\]

We have the following exact sequence 
\begin{equation}
    0 \rightarrow \mathcal{I}_{\Sigma}(1) \rightarrow \mathcal{I}_{\Sigma \cup \Delta}(2) \rightarrow \mathcal{I}_{\Delta/H}(2) \rightarrow 0.
    \label{SigmaDelta}
\end{equation}

Therefore, $p_1, \cdots, p_7$ impose independent conditions on quadrics in $\mathbb{P}^3$ if $\Sigma$ is in general linear position and $\Delta$ impose independent conditions on quadrics in $\mathbb{P}^2$.

\textbf{Case 1}. [$H$ contains 7 points]: 7 points don't impose independent conditions on quadrics in $\mathbb{P}^2$ for dimension reasons, so they don't impose independent conditions on quadrics in $\mathbb{P}^3$. 

\textbf{Case 2}. [$H$ contains 6 points]: 7 points don't impose independent conditions quadrics in $\mathbb{P}^3$ if and only if the 6 points lying on $H$ don't impose independent conditions quadrics in $\mathbb{P}^2$, which happens if and only if these 6 points lie on a plane conic.

\textbf{Case 3}. [$H$ contains 5 points]: Since any 2 points will be in general linear position, 7 points don't impose independent conditions quadrics in $\mathbb{P}^3$ if and only if the 5 points lying on $H$ don't impose independent conditions quadrics in $\mathbb{P}^2$, which happens if and only if 4 of these 5 points are collinear.

\textbf{Case 4}. [$H$ contains 4 points]: By our choice of $H$, any 4 points cannot be collinear. Denote the points lying on $H$ by $p_1, p_2, p_3, p_4$, and without loss of generality $p_1, p_2, p_3$ are not collinear. Furthermore, we may assume that the remaining 3 points are not collinear. In fact, if the remaining 3 points are collinear, we can change $H$ to the plane spanned by $p_4$ and the remaining 3 points. By doing this, the new $H'$ we get has $p_1, p_2, p_3$ not collinear. Any 3 points are in linear general position if they are not collinear. Any 4 points which are not collinear impose independent conditions on quadrics on $\mathbb{P}^2$. Therefore, in this case 7 points will always impose independent conditions on quadrics in $\mathbb{P}^3$. 

\textbf{Case 5}. [$H$ contains 3 points]: By our choice of $H$, any 3 points cannot be collinear and any 4 points cannot be coplanar. Thus the 3 points on $H$ impose independent conditions on quadrics in $\mathbb{P}^2$ and points in $\Sigma$ are in linear general position. So in this case, 7 points will always impose independent conditions on quadrics in $\mathbb{P}^3$.

We have found all the necessary conditions when the 7 points don't impose independent conditions on quadrics in $\mathbb{P}^3$, and it is clear that they are also sufficient.

In summary, all the cases that the seven points in $\mathbb{P}^3$ don't impose independent conditions are the following:\\
$\left(1\right)$ All the seven points are coplanar.\\
$\left(2\right)$ Six of the seven points lie on a plane conic.\\
$\left(3\right)$ Four points are collinear.

If 4 points are collinear, then every quadric vanishing along these 4 points must vanish along the line; if 6 of 7 points lie on a plane conic but no 4 of them are collinear, then every quadric vanishing on these 6 points must vanish along the conic; if all 7 points are coplanar and no 6 of the 7 points lie on a plane conic, then every quadric vanishing on these 7 points must vanish along the plane. In each of these three cases, $Q_1 \cap Q_2 \cap Q_3$ cannot be a complete intersection.

Therefore, we have 

\begin{tikzpicture}
\draw (-3, 0) node{$\Big\{(p_1, \cdots, p_8) \in (\mathbb{P}^3)^8: p_1, \cdots, p_8 \text{ is the complete intersection of 3 quadrics in } \mathbb{P}^3 \Big\}$};
\draw (-2.3, -0.5) -- (-2.3, -1.5);
\draw (-2.2, -0.5) -- (-2.2, -1.5);
\draw (-3.6, -2.1) node{$(p_1, \cdots, p_8) \in (\mathbb{P}^3)^8: p_1, \cdots, p_8 \text{ is the complete intersection  of 3 quadrics}   $};
%\draw (-1.5, -2.5) node{    };
\draw (-3, -3) node{  in $\mathbb{P}^3$ and $p_1, \cdots, p_7$ impose independent conditions on quadrics };
\draw (4.4, -2.5) node{$\overset{open}\subset(\mathbb{P}^3)^7$};
\draw (3.3, -2.5) node{$\Bigg\}$};
\draw (-10.4, -2.5) node{$\Bigg\{$};
\end{tikzpicture}

The final open embedding sends $(p_1, \cdots, p_8)$ to $(p_1, \cdots, p_7) \in \mathbb{P}^3$. This is an open embedding because $p_8$ is uniquely determined by $p_1, \cdots, p_7$.
\end{proof}

Take $G \subset PGL_5$ to be the stabilizer of the hyperplane. After a proper choice of coordinates, $G$ is a subgroup of $PGL_5$ consisting of matrices of the form 
$$ 
\begin{bNiceArray}{c|cccc}[margin,columns-width=auto]
  1 & 0  & 0 & 0 & 0\\
\hline
\text{*} & \Block{4-4}<\Large>{GL_4} &&& \\
\text{*} \\ 
\text{*} \\ 
\text{*}

\end{bNiceArray}.$$

We denote the universal bundle over $BG$ by $\mathcal{F}'$. Observe that $\mathcal{G}^\vee$ defined after the exact sequence \eqref{nexact} corresponds to the subbundle $\mathcal{F} \subset \mathcal{F}'$, which is generated by the constant sections $(0,1,0,0,0),(0,0,1,0,0),(0,0,0,1,0),(0,0,0,0,1)$. Therefore, we have the commutative diagram 
\[\begin{tikzcd}
\mathbb{P}\mathcal{F} \arrow[rr,hookrightarrow,"i"] \arrow[dr,"\gamma"'] && \mathbb{P}\mathcal{F}' \arrow[dl, "\gamma'"] \\
& BG
\end{tikzcd}\]

Take the projective bundle $\mathbb{P}\mathcal{F}$ over $BG$ and its fiber product $(\mathbb{P}\mathcal{F})^7$ over $BG$. We have the composition map $(\mathbb{P}\mathcal{F})^7 \xrightarrow{\eta_i} \mathbb{P}\mathcal{F} \xrightarrow{\gamma} BG$, where $\eta_i$ is the $i$-th projection map and $\gamma$ is the canonical map. Note that by the definition of the fiber product, we have $\gamma \circ \eta_i = \gamma \circ \eta_j$ for any $i,j$. Thus we denote $\gamma \circ \eta_i$ by $\overline{\gamma}$. Moreover, we have the natural map $b_{\omega}: \mathcal{M}_\omega \rightarrow (\mathbb{P}\mathcal{F})^7 $, mapping the curve with eight marked points to the first seven points which lie on a hyperplane in $\mathbb{P}^4$.

Consider the following evaluation map:
\begin{equation}
 \overline{\gamma}^*\gamma'_*\mathcal{O}_{\mathbb{P}\mathcal{F}'}(2) \rightarrow \bigoplus^7_{i=1}\eta_i^*i^*\mathcal{O}_{\mathbb{P}\mathcal{F}'}(2)
 \label{eva2}
\end{equation}
 
Define $V'$ to be the open locus in $(\mathbb{P}\mathcal{F})^7$ such that the evaluation map \eqref{eva2} is surjective. By Lemma \ref{lemma31}, we have that the image of $b_\omega$ is contained in $V'$.
From now on, we consider the evaluation map \eqref{eva2} over $V'$.\\
Let $\mathcal{E}$ be the kernel of the map \eqref{eva2}. Since \eqref{eva2} is surjective, we know that $\mathcal{E}$ is a vector bundle. 
\begin{prop}
We have $\mathcal{M}_\omega \overset{\text{open}}\subset G(3,\mathcal{E})$.    
\end{prop}
\begin{proof} We first construct a map from $\mathcal{M}_\omega$ to $G(3,\mathcal{E})$ by the universal property of the Grassmannian.  Recall that we have the composition maps:
\[\begin{tikzcd}
\mathcal{M}_\omega \arrow[r,"b_\omega"] &(\mathbb{P}\mathcal{F})^7 \arrow[r,"\eta_i"] &\mathbb{P}\mathcal{F} \arrow[rr,hookrightarrow,"i"]\arrow[dr,"\gamma"'] &&\mathbb{P}\mathcal{F}'\arrow[dl,"\gamma'"]\\
&&&BG
\end{tikzcd}\]
We then consider the following diagram:
\[\begin{tikzcd}
0 \arrow[r] & b_\omega^*\mathcal{E} \arrow[r] & b_\omega^*\overline{\gamma}^*\gamma'_*\mathcal{O}_{\mathbb{P}\mathcal{F}'}(2) \arrow[r]\arrow[d,equal] & b_\omega^*{\displaystyle \bigoplus_{i=1}^7 \eta_i^*i^*\mathcal{O}_{\mathbb{P}\mathcal{F}'}(2)} \arrow[r] & 0 \\
0 \arrow[r] & \mathcal{S} \arrow[r] & Sym^2(f_*\omega_f \otimes \mathcal{L}^\vee) \arrow[r] & f_*\left((\omega_f \otimes f^*\mathcal{L}^\vee)^{\otimes 2}\right) \arrow[r] \arrow[u, "\text{surjective}"'] & 0
\end{tikzcd}\]
The above diagram shows that $\mathcal{S}$ is a subbundle of $b_\omega^*\mathcal{E}$, and thus gives a map $\mathcal{M}_\omega \rightarrow G(3,\mathcal{E})$. Furthermore, $\mathcal{M}_\omega$ is isomorphic to an open locus $W$ in $G(3,\mathcal{E})$. And the fibers of $W$ are nets of quadrics in $H^0(\mathcal{O}_{\mathbb{P}^4}(2))$, such that the intersection of the basis is a complete intersection. 
\end{proof}

In summary, we have the following composition maps:
\begin{equation}
    \mathcal{M}_\omega \rightarrow G(3,\mathcal{E}) \rightarrow V' \subset (\mathbb{P}\mathcal{F})^7 \xrightarrow{\eta_i} \mathbb{P}\mathcal{F} \xrightarrow{\gamma} BG \xrightarrow{h} BGL_4,
\label{composition}
\end{equation}
where $h$ is induced by the natural group homomorphism $G \rightarrow GL_4$. Furthermore, $h$ induces an isomorphism between $A^*(BG)$ and $A^*(BGL_4)$.
By our construction, the composition map $\mathcal{M}_\omega  \rightarrow BGL_4$ corresponds to the vector bundle $\mathcal{G}^\vee$ on $\mathcal{M}_\omega$.
%\medbreak
\begin{coro}
    From the composition map \eqref{composition}, we have that the Chow ring $A^*(\mathcal{M}_\omega)$ is generated by $c_1(\mathcal{S}), c_2(\mathcal{S}), c_3(\mathcal{S}), c_1(b_\omega^*\eta_i^*\mathcal{O}_{\mathbb{P}\mathcal{F}}(1)) (1\leq i \leq 7)$ and $c_j(\mathcal{G}) (1 \leq j \leq 4)$.
    \label{coro05}
\end{coro}
\medskip

\begin{prop}
   $c_1(\mathcal{L})= 2 \psi_i$ for any $i$. In particular, $c_1(\mathcal{L})$ is tautological and all $\psi_i$ are equal.
   
\end{prop}
\begin{proof} Note that $\omega_f(-\sigma_1-\cdots -\sigma_8)$ is trivial on fibers, thus 
$\omega_f(-\sigma_1-\cdots -\sigma_8)=f^* \mathcal{L}$, where $\mathcal{L}=f_*\omega_f(-\sigma_1-\cdots -\sigma_8)$ as defined after the exact sequence \eqref{exact1}. Denote the divisor in $\mathcal{C}_\omega$ corresponding to the section $\sigma_i$ by $D_i$. We have
\begin{center}
$\begin{aligned}
\mathcal{L}= \sigma_i^*f^*\mathcal{L} &=\sigma_i^*\omega_f(-\sigma_1 - \cdots -\sigma_8) \\ & = \sigma_i^*\omega_f(-\sigma_i) && \left(\text{since }\sigma_i^*\mathcal{\omega}_f(\sigma_j)=0 \text{ for } i \neq j\right)\\& =\sigma_i^* \omega_f \otimes\mathcal{O} (-D_i)|_{D_i} \\&  =\sigma_i^* \omega_f \otimes \mathcal{N}_{D_i}^{\vee} && \left(\text{definition of normal bundle}\right)
\end{aligned}$    
\end{center}

By the conormal sequence for $\mathcal{M}_\omega \overset{\sigma_i}\hookrightarrow \mathcal{C}_\omega \xrightarrow{f} \mathcal{M}_\omega$, we have $ \mathcal{N}_{D_i}^{\vee} = \sigma_i^* \omega_f$. Therefore, we get $c_1(\mathcal{L}) = c_1(\sigma_i^*\omega_f \otimes \mathcal{N}_{D_i}^{\vee})=c_1\left((\sigma_i^*\omega_f)^{\otimes 2}\right)=2\psi_i$. 
\end{proof}

\begin{coro}
    The Chern classes $c_1(\mathcal{G}), c_2(\mathcal{G}), c_3(\mathcal{G}), \text{ and } c_4(\mathcal{G})$ are tautological.
    \label{coro06}
\end{coro}
\begin{proof} By the exact sequence \eqref{exact1} and Whitney's Formula, we have $c(\mathcal{G})=\frac{c(f_*\omega_f)}{c(\mathcal{L})}$. Thus all the Chern classes of $\mathcal{G}$ are tautological. 
\end{proof}

\begin{lemma}
The classes $c_1(\mathcal{S}), c_2(\mathcal{S}), c_3(\mathcal{S})$ are tautological.
\label{lemma05}
\end{lemma}

\begin{proof}
 Recall that we have the exact sequence $$0 \rightarrow \mathcal{S} \rightarrow Sym^2(f_*\omega_f\otimes \mathcal{L}^\vee) \rightarrow f_*\left((\omega_f \otimes f^*\mathcal{L}^\vee)^{\otimes 2}\right) \rightarrow 0.$$ Observe that $f_*(\omega_f^{\otimes 2})$ is tautological by Grothendieck Riemann–Roch. Using push-pull formula, we know $f_*\left((\omega_f \otimes f^*\mathcal{L}^\vee)^{\otimes 2}\right)$ is tautological. Our lemma then follows from Whitney's Formula. 
 \end{proof}

\begin{lemma}
The classes $c_1(b_\omega^*\eta_i^*\mathcal{O}_{\mathbb{P}\mathcal{F}}(1)) (1\leq i \leq 7)$ are tautological.
\label{lemma06}   
\end{lemma}
\begin{proof} Consider the following diagram:
\[\begin{tikzcd}[row sep=huge]
    &\mathbb{P}\mathcal{G}'^\vee \arrow[d,hook,"\iota'"']\\
\mathcal{C}_\omega \arrow[r,hook,"\iota"] \arrow[rd,"f"] & \mathbb{P}f_*\omega_f^\vee \arrow[d] & \mathbb{P}\mathcal{G}^\vee \arrow[ul,"\simeq","g"']  \arrow[r] \arrow[dl,bend left = 15, "a"]& \mathbb{P}\mathcal{F} \arrow[dl,"\gamma"] \\
& \mathcal{M}_\omega \arrow[ur,bend left=15, "\eta_i'"] \arrow[r]& BGL_4
\end{tikzcd}\]
where $\iota'$ is induced by the exact sequence \eqref{exact1}, $a$ is the structure map and $g$ is the isomorphism between $\mathbb{P}\mathcal{G}^\vee$ and $\mathbb{P}\mathcal{G}'^\vee$ induced by $\mathcal{L}^\vee \otimes \mathcal{G}' \cong \mathcal{G}$.
We have
\begin{center}
$
\begin{aligned}
    b_\omega^*\eta_i^*\mathcal{O}_{\mathbb{P}\mathcal{F}}(1) &= \eta_i'^*\mathcal{O}_{\mathbb{P}\mathcal{G}^\vee}(1)\\& = \eta_i'^*(g^*\mathcal{O}_{\mathbb{P}\mathcal{G}'^\vee}(1)\otimes a^* \mathcal{L})\\& = \sigma_i^*\omega_f \otimes \mathcal{L}.
\end{aligned}
$    
\end{center}

Therefore, we have $c_1\left(b_\omega^*\eta_i^*\mathcal{O}_{\mathbb{P}\mathcal{F}}(1)\right)=c_1(\sigma_i^*\omega_f)+c_1(\mathcal{L})$. In particular, $c_1(b_\omega^*\eta_i^*\mathcal{O}_{\mathbb{P}\mathcal{F}}(1))=3\psi_i$ is tautological for each $i$. \end{proof}

\begin{coro}
    The Chow ring $A^*(\mathcal{M}_{5,8})$ is tautological and $\mathcal{M}_{5,8}$ has the CKgP.
\end{coro}
\begin{proof} The first statement follows from Corollary \ref{coro25}, Corollary \ref{coro05}, Corollary \ref{coro06}, Lemma \ref{lemma05} and Lemma \ref{lemma06}. The second statement follows from Lemmas 3.3, 3.4, 3.5, 3.7 and 3.8 in \cite{HS}. 
\end{proof}

\section{Classes supported on $(\mathcal{M}_{5,9} \setminus \mathcal{M}_{5,9}^3)\setminus U_9$}

Define $\mathcal{M}_{\omega,i}$ to be the locus where $f_*(\omega_f) \rightarrow \displaystyle{\bigoplus_{j\neq i} \sigma_j^*\omega_f}$ drops rank. By Proposition \ref{prop 0.2}, we have $\mathcal{M}_{5,9} \setminus \mathcal{M}_{5,9}^3=U_9 \cup \mathcal{M}_{\omega,1} \cdots \cup \mathcal{M}_{\omega,9}$. Furthermore, we claim the loci $\mathcal{M}_{\omega,i}$ are disjoint. Indeed, without loss of generality, assume for sake of contradiction that $\mathcal{M}_{\omega,1} \cap \mathcal{M}_{\omega,2} \neq \varnothing$ and $C$ is a curve in their intersection. Then $\mathcal{O}(p_1+ p_3+ \cdots + p_9) \cong \omega_C \cong \mathcal{O}(p_2+ \cdots + p_9)$. Thus we have $p_1=p_2$, which contradicts to the fact that the marked points are disjoint.

Each $[\mathcal{M}_{\omega,i}]$ is the pullback of $[\mathcal{M}_{\omega}]$ under the map $\mathcal{M}_{5,9} \rightarrow \mathcal{M}_{5,8}$ forgetting the $i$-th marked point. Proposition \ref{prop04} combined with the fact that pullbacks of tautological classes along forgetful maps are tautological implies $[\mathcal{M}_{\omega,i}]$ is tautological too.

Consider the following exact sequence:
\begin{equation}
    \bigoplus_{i=1}^{9} A^*(\mathcal{M}_{\omega,i}) \rightarrow A^*(\mathcal{M}_{5,9} \setminus \mathcal{M}_{5,9}^3) \rightarrow A^*(U_9) \rightarrow 0
\label{E}
\end{equation}
%\[\begin{tikzcd}
%A^*(\mathcal{M}_{\omega_9}) \arrow[r] &A^*((\mathcal{M}_{5,9} \setminus \mathcal{M}_{5,9}^3)\setminus \mathcal{M}_{\omega_8}') \arrow[r]\arrow[dr,equal] &A^*(U_9) \arrow[r] &0 \\ A^*(\mathcal{M}_{\omega_8}) \arrow[r] &A^*((\mathcal{M}_{5,9} \setminus \mathcal{M}_{5,9}^3)\setminus \mathcal{M}_{\omega_7}') \arrow[r] &A^*((\mathcal{M}_{5,9} \setminus \mathcal{M}_{5,9}^3)\setminus \mathcal{M}_{\omega_8}') \arrow[r] &0 \\ & \vdots \\A^*(\mathcal{M}_{\omega_2}) \arrow[r] &A^*((\mathcal{M}_{5,9} \setminus \mathcal{M}_{5,9}^3)\setminus \mathcal{M}_{\omega_1}') \arrow[r]\arrow[dr,equal] &A^*((\mathcal{M}_{5,9} \setminus \mathcal{M}_{5,9}^3)\setminus \mathcal{M}_{\omega_2}') \arrow[r] &0 \\A^*(\mathcal{M}_{\omega_1}) \arrow[r] &A^*(\mathcal{M}_{5,9} \setminus \mathcal{M}_{5,9}^3) \arrow[r] &A^*((\mathcal{M}_{5,9} \setminus \mathcal{M}_{5,9}^3)\setminus \mathcal{M}_{\omega_1}') \arrow[r] &0 \label{E}\end{tikzcd}\]
It then remains to prove that all classes supported on $\mathcal{M}_{\omega,i}$ are tautological for all $i$, and for this, we are going to use the same method as in the case $n=8$ in Section 3. Furthermore, by symmetry, it suffices to show all classes supported on $\mathcal{M}_{\omega,9}$ are tautological.

Using the same notation as in Section 3, we have the exact sequence \begin{equation}
0 \rightarrow \mathcal{O}_{\mathcal{M}_{\omega,9}} \rightarrow \mathcal{L}^\vee \otimes f_*\omega_f \rightarrow \mathcal{G} \rightarrow 0
\label{15}
\end{equation}
where $\mathcal{G}:=\mathcal{L}^\vee \otimes \mathcal{G}'$.
Take $G \subset PGL_5$ to be the stabilizer of the pair $(H,p_9)$, where $H$ is the hyperplane spanned by the first 8 points. After a proper choice of coordinates, $G$ is the subgroup of $PGL_5$ consisting of matrices of the form 
$$ 
\begin{bNiceArray}{c|cccc}[margin,columns-width=auto]
  1 & 0 & 0 & 0 & 0\\
\hline
0 & \Block{4-4}<\Large>{GL_4} &&& \\
0 \\ 
0 \\
0

\end{bNiceArray}.$$

We denote the universal bundle over $BG$ by $\mathcal{W}'$. Note that $\mathcal{G}^\vee$ defined after the exact sequence \eqref{15} corresponds to the subbundle $\mathcal{W} \subset \mathcal{W}'$, which is generated by the constant sections $(0,1,0,0,0),(0,0,1,0,0), (0,0,0,1,0),(0,0,0,0,1)$. Therefore, we have the commutative diagram 
\[\begin{tikzcd}
\mathbb{P}\mathcal{W} \arrow[rr,hookrightarrow,"i"] \arrow[dr,"\gamma"'] && \mathbb{P}\mathcal{W}' \arrow[dl, "\gamma'"] \\
& BG.
\end{tikzcd}\]

Take the projective bundle $\mathbb{P}\mathcal{W}$ over $BG$ and its fiber product $(\mathbb{P}\mathcal{W})^7$ over $BG$. We have the composition map $(\mathbb{P}\mathcal{W})^7 \xrightarrow{\eta_i} \mathbb{P}\mathcal{W} \xrightarrow{\gamma} BG$, where $\eta_i$ is the $i$-th projection map and $\gamma$ is the canonical map. Note that by the definition of the fiber product, we have $\gamma \circ \eta_i = \gamma \circ \eta_j$ for any $i,j$. Thus we denote $\gamma \circ \eta_i$ by $\overline{\gamma}$. Moreover, we have the natural map $b_{9}: \mathcal{M}_{\omega,9} \rightarrow (\mathbb{P}\mathcal{W})^7 $, mapping the curve with nine marked points to the first seven points which lie on a hyperplane in $\mathbb{P}^4$. We also have the constant section $\sigma_9':BG \rightarrow \mathbb{P}\mathcal{W}'$ corresponding to the fixed 9-th marked point.

Therefore, we have the composition maps
\[\begin{tikzcd}
\mathcal{M}_{\omega,9} \arrow[r,"b_9"] &(\mathbb{P}\mathcal{W})^7 \arrow[r,"\eta_i"]\arrow[dr] &\mathbb{P}\mathcal{W} \arrow[r,hookrightarrow,"i"] \arrow[d,"\gamma"] & \mathbb{P}\mathcal{W}' \arrow[dl,"\gamma'"']\\
&& BG \arrow[ur,bend right, "\sigma_9'"']
\end{tikzcd}\]

Since the first seven marked points and the 9-th marked point impose independent condition on quadrics in $\mathbb{P}^4$, the evaluation map is surjective and its kernel $\mathcal{E}$ is a vector bundle. Therefore, we have the exact sequence
\begin{equation}
0 \rightarrow \mathcal{E} \rightarrow \overline{\gamma}^*\gamma'_* \mathcal{O}_{\mathbb{P}\mathcal{W}'}(2) \xrightarrow{\text{evaluation map}} \bigoplus^7_{i=1}\eta_i^*i^*\mathcal{O}_{\mathbb{P}\mathcal{W}'}(2)\bigoplus\overline{\gamma}^*\sigma_9^*\mathcal{O}_{\mathbb{P}\mathcal{W}'}(2) \rightarrow 
0.
\end{equation}
Since the curve is canonically embedded in $\mathbb{P}^4$, we have the exact sequence
\begin{equation}
0 \rightarrow \mathcal{S} \rightarrow Sym^2(f_*\omega_f \otimes \mathcal{L}^\vee) \rightarrow f_*\left((\omega_f \otimes (f^*\mathcal{L}^\vee))^{\otimes 2}\right) \rightarrow 0.   
\end{equation}

\begin{prop}
We have $\mathcal{M}_{\omega,9} \overset{\text{open}}\subset G(3,\mathcal{E})$.  
\end{prop}
\begin{proof} As in the case $n=8$, we have $b_9^*\overline{\gamma}
^*\gamma'_*\mathcal{O}_{\mathbb{P}\mathcal{W}'}(2)=Sym^2(f_*\omega_f\otimes \mathcal{L}^\vee)$ and the map $f_*\left((\omega_f \otimes (f^*\mathcal{L}^\vee))^{\otimes 2}\right) \rightarrow b_9^*\left(\bigoplus^7_{i=1}\eta_i^*i^*\mathcal{O}_{\mathbb{P}\mathcal{W}'}(2)\bigoplus\overline{\gamma}^*\sigma_9^*\mathcal{O}_{\mathbb{P}\mathcal{W}'}(2)\right)$ is surjective. This proposition then follows from universal property of the Grassmannian. 
\end{proof}

From our argument above, we have the composition map 
$$\mathcal{M}_{\omega,9} \subset G(3, \mathcal{E}) \rightarrow U' \subset (\mathbb{P}\mathcal{W})^7 \rightarrow \mathbb{P}\mathcal{W} \rightarrow BG,$$
where each collection of points in the open set $U'$ is seven points which impose
independent condition on spaces of quadrics in $\mathbb{P}^3$.

\begin{coro}
    The Chow Ring $A^*(\mathcal{M}_{\omega,9})$ is generated by $c_1(\mathcal{S}), c_2(\mathcal{S}), c_3(\mathcal{S}), c_1(\mathcal{W})$, $c_2(\mathcal{W}), c_3(\mathcal{W}), c_4(\mathcal{W}), c_1(\eta_i^*\mathcal{O}_{\mathbb{P}\mathcal{V}'}(1))$ for $1 \leq i \leq 7$. Furthermore, using the same method as in Section 3 we can prove that these classes are tautological.
\label{coro41}
\end{coro}

\begin{prop}
    The Chow Ring $A^*(\mathcal{M}_{5,9})$ is tautological and $\mathcal{M}_{5,9}$ has the CKgP.
\end{prop}
\begin{proof} The first statement follows from the exact sequence \eqref{E}, Corollary \ref{coro25}, and Corollary \ref{coro41}. The second statement follows from Lemmas 3.3, 3.4, 3.7 and 3.8 in \cite{HS}. 
\end{proof}


\begin{thebibliography}{12}
\bibitem{3264}
D. Eisenbud and J. Harris,
\textit{3264 and all that --- a second course in algebraic geometry}, Cambridge University Press, Cambridge, 2016. MR3617981

\bibitem{HS}
S. Canning and H. Larson, \textit{On the Chow and Cohomology Rings of Moduli Space of Stable Curves}, 
J. Eur. Math. Soc. (2024).

\bibitem{HS2}
S. Canning and H. Larson, \textit{The Rational Chow Rings of Moduli Spaces of Hyperelliptic Curves with marked Points}, Ann. Sc. Norm. Super. Pisa Cl. Sci. (5) \textbf{25} (2024), no. 4, 2201–2239, MR4873516 

\bibitem{M2}
D. Mumford, \textit{Towards an enumerative geometry of the moduli space of curves}, Arithmetic and geometry, Vol. II, Progr. Math., vol. 36, Birkhäuser Boston, Boston, MA, 1983, pp. 271–328. MR 717614

\bibitem{M2i}
E. Larson, \textit{The Integral Chow Ring of $\overline{\mathcal{M}}_2$}, Algebr. Geom. \textbf{8} (2021), no. 3, 286–318. MR 4206438

\bibitem{M3}
C. Faber, \textit{Chow rings of moduli spaces of curves. I. The Chow ring of $\overline{\mathcal{M}}_3$}, Ann. of Math. (2) \textbf{132} (1990),
no. 2, 331–419. MR 1070600


\bibitem{M4}
C. Faber, \textit{Chow rings of moduli spaces of curves. II. Some results on the Chow ring of $\overline{\mathcal{M}}_4$}, Ann. of Math. (2) \textbf{132} (1990), no. 3, 421–449. MR 1078265


\bibitem{M5}
E. Izadi, \textit{The Chow ring of the moduli space of curves of genus 5}, The moduli space of curves (Texel Island, 1994), Progr. Math., vol. 129, Birkhäuser Boston, Boston, MA, 1995, pp. 267–304. MR 1363060

\bibitem{M6}
N. Penev and R. Vakil, \textit{The Chow ring of the moduli space of curves of genus six},  Algebr. Geom.
\textbf{2} (2015), no. 1, 123–136. MR 3322200

\bibitem{M789}
S. Canning and H. Larson, \textit{The Chow rings of the moduli spaces of curves of genus 7, 8, and 9}, J. of Algebraic Geom. \textbf{33} (2024), no.1, 55–116, MR4693574

\bibitem{T}
C. Faber, \textit{Relative maps and tautological classes}, J. Eur. Math. Soc. (JEMS) \textbf{7} (2005), no. 1, 13–49. MR2120989

\bibitem{MN}
M. Noether, \textit{\"Uber die invariante Darstellung algebraicher Funktionen}, Math. Ann., 17 (1880)
263–284.

\end{thebibliography}
\end{document}